\documentclass[usletter, 10pt, conference]{ieeeconf}
\IEEEoverridecommandlockouts
\usepackage{epsfig,psfrag}
\usepackage{graphicx}
\usepackage{amsmath}
\usepackage{mathtools}
\usepackage{enumerate}
\usepackage{amssymb}
\usepackage{here}
\usepackage{cite}
\usepackage{etoolbox}
\usepackage{color}
\usepackage{bbm}
\usepackage{mathrsfs}
\usepackage{datetime}
\usepackage{dsfont,accents}
\usepackage{url}

\def\E{\mathbb{E}}

\everymath{\displaystyle}

\AtBeginEnvironment{bmatrix}{\setlength{\arraycolsep}{5pt}}

\newtheorem{theorem}{Theorem}[section]

\newtheorem{lemma}[theorem]{Lemma}
\newtheorem{define}[theorem]{Definition}

\newtheorem{remark}[theorem]{Remark}

\newcommand{\mendth}{\hfill \ensuremath{\vartriangle}}

\DeclareMathOperator*{\col}{col}

\DeclareMathOperator*{\rank}{rank}

\DeclareMathOperator{\sgn}{sgn}

\def\E{\mathbb{E}}

\def\d{\textnormal{d}}

\providecommand\X[1]{\boldsymbol{X_{#1}}}
\providecommand\Z[1]{\boldsymbol{Z_{#1}}}

\providecommand\phib{\boldsymbol{\emptyset}}
\providecommand\Deltab{\boldsymbol{\Delta}}

\def\llongrightarrow{\relbar\joinrel\relbar\joinrel\relbar\joinrel\rightarrow}

\providecommand{\rarrow}[1]{\stackrel{#1}{\llongrightarrow}}

\def\Xz{\boldsymbol{X}}
\def\Zz{\boldsymbol{Z}}

\providecommand\X[1]{\boldsymbol{X_{#1}}}
\providecommand\Z[1]{\boldsymbol{Z_{#1}}}
\providecommand\phib{\boldsymbol{\emptyset}}

\title{Robust ergodicity and tracking in antithetic integral control of stochastic biochemical reaction networks}

\author{Corentin Briat and Mustafa Khammash
\thanks{Corentin Briat and Mustafa Khammash are with the Department of Biosystems Science and Engineering, ETH-Z\"{u}rich, Switzerland; email: {\tt  \{corentin.briat,mustafa.khammash\}@bsse.ethz.ch}; url: \protect\url{http://www.bsse.ethz.ch/ctsb}, \protect\url{http://www.briat.info}}}

\selectfont

\begin{document}
\maketitle

\begin{abstract}
Controlling stochastic reactions networks is a challenging problem with important implications in various fields such as systems and synthetic biology. Various regulation motifs have been discovered or posited over the recent years, the most recent one being the so-called Antithetic Integral Control (AIC) motif \cite{Briat:15e}. Several favorable properties for the AIC motif have been demonstrated for classes of reaction networks that satisfy certain irreducibility, ergodicity and output controllability conditions. Here we address the problem of verifying these conditions for large sets of reaction networks with fixed topology using two different approaches. The first one is quantitative and relies on the notion of interval matrices while the second one is qualitative and is based on sign properties of matrices. The obtained results lie in the same spirit as those obtained in \cite{Briat:15e} where properties of reaction networks are independently characterized in terms of control theoretic concepts, linear programming conditions and graph theoretic conditions.
\end{abstract}


\section{Introduction}

Homeostasis \cite{Banci:13}  is the ability of living organisms to adapt to external and dynamical stimuli. At the cellular level, homeostasis can be physiologically achieved through \emph{perfect adaptation}\cite{Berg:15}, a mechanism ensuring properties in living cells analogous to robustness and disturbance rejection in control systems. Several homeostatic motifs achieving perfect or near-perfect adaptation have been discovered and proposed over the past years. Important examples are the \emph{incoherent feedforward motif} and the \emph{negative feedback loop motif}; see e.g. \cite{Alon:07,Ma:09,Ferrell:16}. It is notably shown in \cite{Yi:00} that bacterial chemotaxis involves integral feedback and, as pointed out in \cite{Ma:09}, that this integral feedback is implemented as a negative feedback loop motif with a buffering node. More recently, two novel integral control motifs have been proposed: the \emph{antithetic integral control motif} \cite{Briat:15e} and an \emph{autocatalytic integral control motif} \cite{Briat:16a}. The antithetic integral motif benefits from several interesting properties: it can achieve tracking and perfect adaptation for the controlled network, its metabolic load can be made arbitrarily close to the constitutive bound\footnote{The constitutive bound is defined here as the metabolic cost of a constitutive open-loop controller that would achieve the same steady-state.} via a suitable choice for its parameters, and it can be used in both deterministic and stochastic settings. It is also the first homeostatic motif that is proved to work in the stochastic setting, implying then that it can perfectly achieve its function in the low copy number regime. An unexpected and intriguing property, which is exclusive to the stochastic setting, is that of \emph{controller innocuousness}: the antithetic integral control motif cannot make the closed-loop network non-ergodic (i.e. make the closed-loop network trajectories to grow without bound). This has to be contrasted with the common understanding that setting a too high gain for the integral action will likely result in a destabilization of the closed-loop system (unless quite restrictive conditions are met by the open-loop system). As this feature is only present in the stochastic setting, we can conclude that this property emerges from the presence of noise in the dynamics and, therefore, that the antithetic integral control exploits this noise for achieving its function. Several other networks, such as the bistable switch of \cite{Tian:06} and the circadian clock model of \cite{Vilar:02}, have been reported to heavily rely upon noise to realize their function -- their deterministic counterparts indeed fail in achieving any similar function

Sufficient conditions for the open-loop network that characterize whether the antithetic integral controller would ensure the desired tracking and adaptation properties for the closed-loop network have been stated in \cite{Briat:15e} for the case of unimolecular reaction networks and a particular class of bimolecular reaction networks. In the unimolecular case, these conditions are 1) the \emph{irreducibility} of the state-space of the closed-loop network, 2) the \emph{output-controllability} of the open-loop network, and 3) the \emph{Hurwitz stability of the characteristic matrix} of the open-loop network. These conditions lie in the same spirit as the conditions obtained in \cite{Briat:13i} for establishing the long-term behavior of stochastic reaction networks where computationally cheap and versatile conditions were reported. The irreducibility of the state-space is tacitly assumed to hold in \cite{Briat:13i,Briat:15e} but can be efficiently checked using the method described in  \cite{Gupta:14}. Interestingly, the irreducibility of the state-space of the closed-loop network, the Hurwitz stability of the characteristic matrix, and the output controllability of the open-loop network can all be cast as a mixture of linear programs and linear algebraic conditions. The complexity of these conditions depends linearly on the number of distinct of species involved in the open-loop network, which makes the approach highly scalable and suitable for considering biological networks consisting of a large number of distinct molecular species.

The conditions in \cite{Briat:15e} are formulated for networks with fixed rate parameters, which is equivalent to having a fixed characteristic matrix in the unimolecular case. The objective of the current paper is to extend these results to families of characteristic matrices that arise from uncertainties at the level of the rate parameters while, at the same time, preserving the scalability of the approach. We show here that this can be achieved in, at least, two different ways. The first is quantitative and assumes that the characteristic matrix is an \emph{interval-matrix}, i.e. the entries of the characteristic matrix lie within some known and fixed intervals. We prove several results in this regard, the most important one assessing that all the matrices in the interval are Hurwitz stable if and only if the upper-bound is so, and that all the matrices in the interval are output-controllable if and only if the lower-bound is so. This readily translates to a linear program having the exact same complexity has the one stated in \cite{Briat:15e} dealing with the ergodicity and output-controllability of a network with fixed rates. The second approach is qualitative and assumes that only the sign pattern of the characteristic matrix is known. In this case, we are interested in assessing properties for all the matrices sharing the same sign pattern -- an approach often referred to as \emph{qualitative analysis}. Several results obtained in \cite{Briat:14c} are revisited and extended to the current problem. The main result states that, given a known sign pattern, all the characteristic matrices sharing this sign pattern are Hurwitz stable and output-controllable if and only if the digraph associated with the sign pattern is acyclic and has a path connecting the input node to the output node. As in the nominal and the robust cases, these conditions can be recast as computationally inexpensive linear programs that can be applied to large scale networks. 

\noindent\textbf{Outline.} We recall in Section \ref{sec:prel} several definitions and results related to reaction networks and antithetic integral control. Section \ref{sec:rob} is concerned with the extension of the results in \cite{Briat:15e} to characteristic interval-matrices whereas Section \ref{sec:struct} addresses the structural case where only the sign pattern of the characteristic matrix is known. An illustrative example is finally considered in Section \ref{sec:ex}.

\noindent\textbf{Notations.} The set of positive (nonnegative) vectors is defined as $\mathbb{R}_{>0}^d$ ($\mathbb{R}_{\ge0}^d$). The natural basis of $\mathbb{R}^n$ is denoted by $\{e_i\}_{i=1}^n$, the set of nonnegative integers, natural numbers and integers are denoted by $\mathbb{Z}_{\ge0}$, $\mathbb{Z}_{>0}$ and $\mathbb{Z}$, respectively.

\section{Preliminaries}\label{sec:prel}

\subsection{Reaction networks}\label{sec:RN}

We consider here a reaction network with $d$ molecular species $\X{1},\ldots,\X{d}$ that interacts through $K$ reaction channels $\mathcal{R}_1,\ldots,\mathcal{R}_K$ formulated as
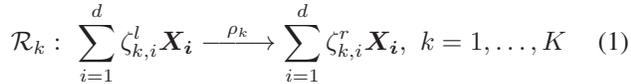
\begin{equation}
 \mathcal{R}_k:\ \sum_{i=1}^d\zeta_{k,i}^l\X{i}\rarrow{\rho_k}  \sum_{i=1}^d\zeta_{k,i}^r\X{i},\ k=1,\ldots,K
\end{equation}
where $\rho_k\in\mathbb{R}_{>0}$ is the rate and $\zeta_{k,i}^l,\zeta_{k,i}^r\in\mathbb{Z}_{\ge0}$. Each reaction is described by a stoichiometric vector and a propensity function. The stoichiometric vector of reaction  $\mathcal{R}_k$ is denoted by $\zeta_k:=\zeta_k^r-\zeta_k^l\in\mathbb{Z}^d$ where $\zeta_k^r=\col(\zeta_{k,1}^r,\ldots,\zeta_{k,d}^r)$ and $\zeta_k^l=\col(\zeta_{k,1}^l,\ldots,\zeta_{k,d}^l)$. Hence, when the reaction $\mathcal{R}_k$ fires, the state jumps from $x$ to $x+\zeta_k$. We define the stoichiometry matrix $S\in\mathbb{Z}^{d\times K}$ as $S:=\begin{bmatrix}
  \zeta_1\ldots\zeta_K
\end{bmatrix}$. When  the kinetics is  mass-action, the propensity function of reaction $\mathcal{R}_k$ is given by  $\textstyle\lambda_k(x)=\rho_k\prod_{i=1}^d\frac{x_i!}{(x_i-\zeta_{k,i}^l)!}$ and is such that  $\lambda_k(x)=0$ if $x\in\mathbb{Z}_{\ge0}^d$ and $x+\zeta_k\notin\mathbb{Z}_{\ge0}^d$. We denote this reaction network by $(\Xz,\mathcal{R})$.

\subsection{Antithetic integral control}

Antithetic integral control has been introduced in \cite{Briat:15e} for solving the perfect adaptation problem in stochastic reaction networks. The key idea is to adjoin to the open-loop network $(\Xz,\mathcal{R})$ a set of additional species and reactions (the controller) in such a way that, by acting on the production rate of the first molecular species $\X{1}$, referred to as the \emph{actuated species}, we can steer the mean value of the \emph{controlled species} $\X{\ell}$, $\ell\in\{1,\ldots,d\}$, to a desired set-point (the reference). The controller is also required to be able to reject constant disturbances and to accommodate with possible sporadic changes in the rate parameters (perfect adaptation). As proved in \cite{Briat:15e}, the antithetic integral control motif $(\Zz,\mathcal{R}^c)$ defined with
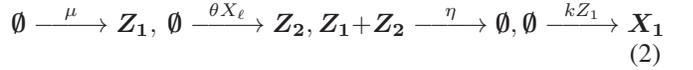
\begin{equation}
  \phib\rarrow{\mu}\Z{1},\ \phib\rarrow{\theta X_\ell}\Z{2}, \Z{1}+\Z{2}\rarrow{\eta}\phib, \phib\rarrow{kZ_1}\X{1}
\end{equation}
solves this control problem with the set-point being equal to $\mu/\theta$. Above, $\Z{1}$ and $\Z{2}$ are the \emph{controller species}, namely, the \emph{actuating species} and the \emph{sensing species}, respectively. The four controller parameters $\mu,\theta,\eta,k>0$ are assumed to be freely assignable to any desired value. The first reaction is the \emph{reference reaction} as it encodes part of the reference value $\mu/\theta$ as its own rate. The second one is the \emph{measurement reaction} that produces the sensing species $\Z{2}$ at a rate proportional to the current population of the controlled species $\X{\ell}$. The third reaction is the \emph{comparison reaction} as it compares the populations of the controller species and annihilates one molecule of each when these populations are both positive. Note also that this reaction is the one that closes the overall (negative) control loop. Finally, the fourth reaction is the \emph{actuation reaction} that produces the actuated species $\X{1}$ at a rate proportional to the actuating species $\Z{1}$.

\subsection{Summary of the results for unimolecular networks}

In the unimolecular reaction networks case, the propensity functions are nonnegative affine functions of the current state of the network. Hence, we can write
\begin{equation}
  \lambda(x)=W(\rho)x+w_0(\rho),
\end{equation}
where $W(\rho)\in\mathbb{R}_{\ge0}^{K\times d}$, $w_0(\rho)\in\mathbb{R}_{\ge0}^{K}$ and $\rho\in\mathbb{R}^{n_\rho}_{>0}$ is the positive vector of reaction rates, as defined in Section \ref{sec:RN}. Before pursuing any further, it is convenient to introduce some terminologies and results.
\begin{define}
  The \emph{characteristic matrix} $A(\rho)$ and the \emph{offset vector} $b_0(\rho)$ of a unimolecular reaction network $(\Xz,\mathcal{R})$ are defined as
  \begin{equation}
    A(\rho):=SW(\rho)\ \textnormal{and}\ b_0:=Sw_0(\rho).
  \end{equation}
\end{define}

A particularity of unimolecular reaction networks is that the matrix $A(\rho)$ is always Metzler; i.e. all the off-diagonal elements are nonnegative. This property plays an essential role in the derivation of the results of \cite{Briat:15e} and will also be essential for the derivation of the main results of this paper.

\begin{define}
  The closed-loop reaction network obtained from the interconnection of the open-loop reaction network $(\Xz,\mathcal{R})$ and the controller network  $(\Zz,\mathcal{R}^c)$ is denoted as $((\Xz,\Zz),\mathcal{R}\cup\mathcal{R}^c)$.
\end{define}

We will also need the following result on the output controllability of linear SISO positive systems:
\begin{lemma}[\cite{Briat:15e}]\label{lem:prel}
  Let $M\in\mathbb{R}^{d\times d}$ be a Metzler matrix. Then, the following statements are equivalent:
  \begin{enumerate}[(a)]
  \item The  linear system
    \begin{equation}\label{eq:posglab}
    \begin{array}{lcl}
              \dot{x}(t)&=&Mx(t)+e_iu(t)\\
              y(t)&=&e_j^Tx(t)
    \end{array}
    \end{equation}
    is output controllable.
    \item $\rank\begin{bmatrix}
      e_j^Te_i & e_j^TMe_i & \ldots & e_j^TM^{d-1}e_i
    \end{bmatrix}=1$.
    %
    %
    %
    \item There is a path from node $i$ to node $j$ in the directed graph $\mathcal{G}_M=(\mathcal{V},\mathcal{E})$ defined with $\mathcal{V}:=\{1,\ldots,d\}$ and
\begin{equation*}
    \mathcal{E}:=\{(m,n):\ e_{n}^TMe_{m}\ne0,\ m,n\in V,\ m\ne n\}.\hfill\mendth
\end{equation*}
  \end{enumerate}
  Moreover, when the matrix $M$ is Hurwitz stable, then the above statements are also equivalent to:
\begin{enumerate}[(a)]
  \setcounter{enumi}{4}
     \item The inequality $e_j^TM^{-1}e_i\ne0$ holds or, equivalently, the static-gain of the system \eqref{eq:posglab} is nonzero.\mendth
\end{enumerate}
\end{lemma}

Before stating the main result, it is convenient to define here the following properties that will be recurrently satisfied whenever the conditions of the main results are satisfied:
{\vspace{4mm}}
\noindent\fbox{
\parbox{0.45\textwidth}{
\begin{itemize}
    \item[\textbf{P1.}] the closed-loop reaction network $((\Xz,\Zz),\mathcal{R}\cup\mathcal{R}^c)$ is ergodic,
    \item[\textbf{P2.}] the mean of the controlled species satisfies $\E[X_\ell(t)]\to\mu/\theta$ as $t\to\infty$,
  \item[\textbf{P3.}] the second-order moment matrix $\E[X(t)X(t)^T]$ is uniformly bounded and globally converges to its unique stationary value.
  \end{itemize}}}

{\vspace{3mm}We can now recall the main result of \cite{Briat:15e} on unimolecular reaction networks:}
\begin{theorem}[\cite{Briat:15e}]\label{th:affine:nominal}
Assume that the open-loop reaction network $(\Xz,\mathcal{R})$ is unimolecular and that the state-space of the closed-loop reaction network $((\Xz,\Zz),\mathcal{R}\cup\mathcal{R}^c)$ is irreducible. Let us also assume that the vector of reaction rates $\rho$ is fixed and equal to some nominal value $\rho_0$. In this regard, we set $A=A(\rho_0)$ and $b_0=b_0(\rho_0)$. Then, the following statements are equivalent:
  \begin{enumerate}[(a)]
    \item There exist vectors $v\in\mathbb{R}_{>0}^d$, $w\in\mathbb{R}_{\ge0}^d$, $w_1>0$, such that $v^TA <0$ and $w^TA +e_\ell^T=0$.
    \item The positive linear system describing the dynamics of the first-order moments given by
    \begin{equation}\label{eq:linpos:nominal}
      \begin{array}{rcl}
        \dfrac{\d\E[X(t)]}{\d t}&=&A \E[X(t)]+e_1u(t)+b_0(\rho_0)\\
        y(t)&=&e_\ell^T\E[X(t)]
      \end{array}
    \end{equation}
    is asymptotically stable and output controllable; i.e. the characteristic matrix $A$ of the network $(\Xz,\mathcal{R})$ is Hurwitz stable and
    \begin{equation}
      \rank\begin{bmatrix}
      e_\ell^Te_1 & e_\ell^TA e_1 & \ldots & e_\ell^TA ^{d-1}e_1
    \end{bmatrix}=1.
    \end{equation}
  \end{enumerate}
  Moreover, when one of the above statements holds, then for any values for the controller rate parameters $\eta,k>0$, the properties \textbf{P1.}, \textbf{P2.} and \textbf{P3.} hold provided that
    \begin{equation}\label{eq:lowerbound}
      \dfrac{\mu}{\theta}>\dfrac{v^Tb_0}{c e_\ell^Tv}
    \end{equation}
   where $c>0$ and $v\in\mathbb{R}_{>0}^d$ verify $v^T(A+cI)\le0$.\mendth
%
\end{theorem}

Interestingly, the conditions stated in the above result can be numerically verified by checking the feasibility of the following linear program
  \begin{equation}\label{feas:nom}
    \begin{array}{rcl}
      \textnormal{Find} && v\in\mathbb{R}_{>0}^d,\ w\in\mathbb{R}_{\ge0}^d\\
      \textnormal{s.t.}&&w^Te_1>0\\
      &&v^TA<0\\
      &&w^TA+e_\ell^T=0
    \end{array}
  \end{equation}
  which involves $2d$ variables, $3d$ inequality constraints and $d$ equality constraints.

\section{Robust ergodicity of the closed-loop network}\label{sec:rob}

\subsection{Preliminaries}

The results obtained in the previous section apply when the characteristic matrix $A=A(\rho_0)$ is fixed as the linear programming problem \eqref{feas:nom} can only be solved for a single and known characteristic matrix $A(\rho)$. The goal is to generalize these results to the case where the characteristic matrix $A(\rho)$ and the offset vector $b_0(\rho)$ are uncertain and belong to the sets
\begin{equation}\label{eq:calA}
  \mathcal{A}:=\left\{M\in\mathbb{R}^{d\times d}:\ A^-\le M\le A^+\right\},\ A^-\le A^+,
\end{equation}
and
\begin{equation}\label{eq:calB}
  \mathcal{B}:=\left\{b\in\mathbb{R}_{\ge0}^{d}:\ b_0^-\le b\le b_0^+\right\},\ 0\le b_0^-\le b_0^+,
\end{equation}
where the inequality signs are componentwise and the extremal Metzler matrices $A^+,A^-$ and vectors $b_0^-,b_0^+$ are known. These bounds can be determined such that the inequalities
\begin{equation}
A^-\le A(\rho)\le A^+\ \textnormal{and}\ b_0^-\le b_0(\rho)\le b_0^+
\end{equation}
hold for all $\rho\in\mathcal{P}\subset\mathbb{R}_{>0}^K$ where $\mathcal{P}$ is the compact set containing all the possible values for the rate parameters. We have the following preliminary result:
\begin{lemma}\label{lem:frob}
The following statements are equivalent:
\begin{enumerate}[(a)]
  \item All the matrices in $\mathcal{A}$ are Hurwitz stable;
  \item The matrix $A^+$ is Hurwitz stable.\mendth
\end{enumerate}
\end{lemma}
\begin{proof}
The proof that (a) implies (b) is immediate. The converse can be proved using the fact that for two Metzler matrices $M_1,M_2\in\mathbb{R}^{d\times d}$ verifying the inequality $M_1\le M_2$, we have that $\lambda_F(M_1)\le\lambda_F(M_2)$ where $\lambda_F(\cdot)$ denotes the Frobenius eigenvalue (see e.g. \cite{Berman:94}). Hence, we have that $\lambda_F(M)\le\lambda_F(A^+)<0$ for all $M\in\mathcal{A}$. The conclusion then readily follows.
\end{proof}

\subsection{Main result}

We are now in position to state the following generalization of Theorem \ref{th:affine:nominal}:
\begin{theorem}\label{th:affine:robust}
Let us consider a unimolecular reaction network $(\Xz,\mathcal{R})$ with characteristic matrix $A$ in $\mathcal{A}$ and offset vector $b_0$ in $\mathcal{B}$. Assume also that the state-space of the closed-loop reaction network $((\Xz,\Zz),\mathcal{R}\cup\mathcal{R}^c)$ is irreducible. Then, the following statements are equivalent:
  \begin{enumerate}[(a)]
   \item All the matrices in $\mathcal{A}$ are Hurwitz stable and for all $A\in\mathcal{A}$, there exists a vector $w\in\mathbb{R}_{\ge0}^d$ such that $w_1>0$ and $w^TA+e_\ell^T=0$.\label{st:rob1}
    \item There exist two vectors $v_+\in\mathbb{R}_{>0}^d$, $w_-\in\mathbb{R}_{\ge0}^d$ such that $v_+^TA^+<0$, $w_-^Te_1>0$ and $w_-^TA^-+e_\ell^T=0$.\label{st:rob2}
  \end{enumerate}
    Moreover, when one of the above statements holds, then for any values for the controller rate parameters $\eta,k>0$ and any $(A,b_0)\in\mathcal{A}\times\mathcal{B}$, the properties \textbf{P1.}, \textbf{P2.} and \textbf{P3.} hold provided that
\begin{equation}\label{eq:lowerbound2}
      \dfrac{\mu}{\theta}>\dfrac{q^T(A^+-\Delta)^{-1}b^+}{cq^T(A^+-\Delta)^{-1}e_\ell}
    \end{equation}
    and
    \begin{equation}\label{eq:lowerbound2b}
      q^T(c(A^+-\Delta)^{-1}+I_d)\ge0
    \end{equation}
for some $c>0$, $q\in\mathbb{R}^d_{>0}$ and for all $\Delta\in[0,A^+-A^-]$.\mendth
%
\end{theorem}
\begin{proof}
%
The proof that \eqref{st:rob1} implies \eqref{st:rob2} is immediate. So let us focus on the reverse implication. Define $A(\Delta):=A^+-\Delta$, $\Delta\in \Deltab:=[0,A^+-A^-]$, where the set membership symbol is componentwise. The Hurwitz-stability of all the matrices in $\mathcal{A}$ directly follows for the theory of linear positive systems and Lemma \ref{lem:frob}. We need now to construct a suitable positive vector $v(\Delta)\in\mathbb{R}_{>0}^d$ such that $v(\Delta)^TA(\Delta)<0$ for all $\Delta\in\Deltab$ provided that $v_+^TA^+<0$. We prove now that such a $v(\Delta)$ is given by $v(\Delta)=(I_d+\Delta(A^+-\Delta)^{-1})^Tv_+$. Since $A(\Delta)=A^+-\Delta$, then we immediately get that $v(\Delta)^TA(\Delta)=v_+^TA^+<0$. Hence, it remains to prove the positivity of the vector $v(\Delta)$ for all $\Delta\in\Deltab$. The difficulty here is that the product $\Delta(A^+-\Delta)^{-1}$ is a nonnegative matrix since $\Delta\ge0$ and $(A^+-\Delta)^{-1}\le0$, the latter being the consequence of the fact that $A^+-\Delta$ is Metzler and Hurwitz stable (see e.g. \cite{Berman:94}). Therefore, there may exist values for $v_+\in\mathbb{R}^d_{>0}$ for which we have $v(\Delta)\ngtr0$. To rule this possibility out, we restrict the analysis to all those $v_+\in\mathbb{R}^d_{>0}$ for which we have $v_+^TA^+<0$. We can parameterize all these $v_+$ as $v_+(q)=-(A^+)^{-T}q$ where $q\in\mathbb{R}^d_{>0}$ is arbitrary. We prove now that the vector $v(\Delta)=-(I_d+\Delta(A^+-\Delta)^{-1})^T(A^+)^{-T}q>0$ is positive for all $q\in\mathbb{R}^d_{>0}$ and for all $\Delta\in \Deltab$. This is done by showing below that the matrix $\mathcal{M}:=-(I_d+\Delta(A^+-\Delta)^{-1})^T(A^+)^{-T}$ is nonnegative and invertible. Indeed, we have that
\begin{equation}
  \begin{array}{lcl}
    \mathcal{M} \hspace{-2mm}&=& \hspace{-2mm}-(I_d+\Delta(A^+-\Delta)^{-1})^T(A^+)^{-T}\\
                            &=&\hspace{-2mm}-\left((A^+)^{-1}+(A^+)^{-T}\Delta(A^+-\Delta)^{-1}\right)^T\\
                            &=& \hspace{-2mm}-\left((A^+)^{-1}+(A^+)^{-T}\Delta(I_d-(A^+)^{-1}\Delta)^{-1}A^+\right)^T\\
                            &=& \hspace{-2mm}-(A^+-\Delta)^{-T}
  \end{array}
\end{equation}
where the latter expression follows from the Woodbury matrix identity. Since $(A^+-\Delta)=A(\Delta)$ is Metzler and Hurwitz stable for all $\Delta\in\Deltab$, then $A^+-\Delta$ is invertible and we have $-(A^+-\Delta)^{-1}\ge0$, which proves the result.

Let us now consider then the output controllability condition and define $A(\Delta)$ as $A(\Delta):=A^-+\Delta$ where $\Delta\in \Deltab$. We use a similar approach as previously and we build a $w(\Delta)$ that verifies the expression $w(\Delta)^TA(\Delta)+e_\ell^T=0$ for all $\Delta\in\Deltab$ provided that $w_-^TA^-+e_\ell^T=0$. We prove that such a $w(\Delta)$ is given by $w(\Delta):=(A^-(A^-+\Delta)^{-1})^Tw_-$. We first prove that this $w(\Delta)$ is nonnegative and that it verifies $e_1^Tw(\Delta)>0$ for all $\Delta\in\Deltab$. To show this, we rewrite this $w(\Delta)$ as $w(\Delta)=(I_d-\Delta(A^-+\Delta)^{-1})^Tw_-$ and using the fact that $(A^-+\Delta)^{-1}\le0$ since $(A^-+\Delta)$ is a Hurwitz stable Metzler matrix for all $\Delta\in\Deltab$, then we can conclude that $w(\Delta)\ge w_-\ge0$ for all $\Delta\in\Deltab$. An immediate consequence is that $w(\Delta)^Te_1\ge w_-^Te_1>0$ for all $\Delta\in\Deltab$. This proves the first part. We now show that this $w(\Delta)$ verifies the output controllability condition. Evaluating then $w(\Delta)^T(A^-+\Delta)$ yields
%
\begin{equation}
  \begin{array}{lcl}
    w(\Delta)^T(A^-+\Delta) \hspace{-2mm}&=&\hspace{-2mm} w_-^T(A^-(A^-+\Delta)^{-1})(A^-+\Delta)\\
                                                    &=&\hspace{-2mm} w_-^TA^-\\
                                                    &=&\hspace{-2mm}-e_\ell^T
  \end{array}
\end{equation}
where the last row has been obtained from the assumption that $w_-^TA^-+e_\ell^T=0$. This proves the second part. Finally, the condition \eqref{eq:lowerbound2} is obtained by substituting the expression for $v(\Delta)$ defined above in \eqref{eq:lowerbound}. This completes the proof.
\end{proof}

As in the nominal case, the above result can be exactly formulated as the linear program
 \begin{equation}\label{feas:rob}
    \begin{array}{rcl}
      \textnormal{Find} && v\in\mathbb{R}_{>0}^d,\ w\in\mathbb{R}_{\ge0}^d\\
      \textnormal{s.t.}&&w^Te_1>0\\
      &&v^TA^+<0\\
      &&w^TA^-+e_\ell^T=0
    \end{array}
  \end{equation}
  which has exactly the same complexity as the linear program \eqref{feas:nom}. Hence, checking the possibility of controlling a family of networks defined by a characteristic interval-matrix is not more complicated that checking the possibility of controlling a single network.
\section{Structural ergodicity of the closed-loop network}\label{sec:struct}

In the previous section, we were interested in uncertain networks characterized in terms of a characteristic interval-matrix. We consider here a different approach based on the qualitative analysis of matrices in which we assume that only the sign-pattern of the characteristic matrix is known. In such a case, we are interested in formulating tractable conditions establishing whether all the characteristic matrices sharing the same sign-pattern verify the conditions of Theorem \ref{th:affine:nominal}.

To this aim, let us consider the set of \emph{sign symbols} $\mathbb{S}:=\{0,\oplus,\ominus\}$ and define a \emph{sign-matrix} as a matrix with entries in $\mathbb{S}$. The qualitative class $\mathcal{Q}(\Sigma)$ of a sign-matrix $\Sigma\in\mathbb{S}^{m\times n}$ is defined as
\begin{equation}
  \mathcal{Q}(\Sigma):=\left\{M\in\mathbb{R}^{m\times n}:\ \sgn(M)=\sgn(\Sigma)\right\}
\end{equation}
where the signum function $\sgn(\cdot)$ is defined as
\begin{equation}
    [\sgn(\Sigma)]_{ij}:=\left\{\begin{array}{ll}
                            1&\textnormal{if}\ \Sigma_{ij}\in\mathbb{R}_{>0}\cup\{\oplus\},\\
                            -1&\textnormal{if}\ \Sigma_{ij}\in\mathbb{R}_{<0}\cup\{\ominus\},\\
                            0&\textnormal{otherwise}.
                            \end{array}\right.
\end{equation}
The following result proved in \cite{Briat:14c} will turn out to be a key ingredient for deriving the main result of this section:
\begin{lemma}[\cite{Briat:14c}]\label{lem:struct}
  Let $\Sigma\in\mathbb{S}^{d\times d}$ be a given Metzler sign-matrix. Then, the following statements are equivalent:
  \begin{enumerate}[(a)]
    \item All the matrices in $\mathcal{Q}(\Sigma)$ are Hurwitz stable.
    \item The matrix $\sgn(\Sigma)$ is Hurwitz stable.
    %
    %
    \item The diagonal elements of $\Sigma$ are negative and the directed graph $\mathcal{G}_\Sigma=(\mathcal{V},\mathcal{E})$ defined with
    \begin{itemize}
      \item $\mathcal{V}:=\{1,\ldots,d\}$ and
      \item $\mathcal{E}:=\{(m,n):\ e_n^T\Sigma e_m\ne0,\ m,n\in V,\ m\ne n\}$.
    \end{itemize}
    is an acyclic directed graph.\mendth
  \end{enumerate}
\end{lemma}
We are now ready to state the main result of this section:
\begin{theorem}\label{th:affine:structural}
   Let $S_A\in\mathbb{S}^{d\times d}$ be Metzler and $S_b\in\{0,\oplus\}^{d}$ be some given sign patterns for the characteristic matrix and the offset vector of some unimolecular  reaction network $(\Xz,\mathcal{R})$. Assume that $\ell\ne1$ and that the state-space of the closed-loop reaction network $((\Xz,\Zz),\mathcal{R}\cup\mathcal{R}^c)$ is irreducible. Then, the following statements are equivalent:
  \begin{enumerate}[(a)]
      \item All the matrices in $\mathcal{Q}(S_A)$ are Hurwitz stable and, for all $A\in\mathcal{Q}(S_A)$, there exists a vector $w\in\mathbb{R}_{\ge0}^d$ such that $w_1>0$ and $w^TA+e_\ell^T=0$.\label{st:st:3}
          \item The diagonal elements of $S_A$ are negative and the directed graph $\mathcal{G}_{S_A}$ is acyclic and contains a path from node $1$ to node $\ell$.\label{st:st:2}
    \item There exist vectors $v_1\in\mathbb{R}_{>0}^d$, $v_2,v_3\in\mathbb{R}_{\ge0}^{d}$, $||v_2+v_3||_1=1$, such that  the conditions\label{st:st:1}
    \begin{equation}
        v_1^T\sgn(S_A)<0
    \end{equation}
    and
    \begin{equation}
      v_2-\sgn(\sgn(S_A)+e_1e_\ell^T)v_3=0
    \end{equation}
    hold.
    %
        %
    %
        %
  \end{enumerate}
    Moreover, when one of the above statements holds, then for any values for the controller rate parameters $\eta,k>0$ and any $(A,b_0)\in\mathcal{Q}(S_A)\times\mathcal{Q}(S_b)$, the properties \textbf{P1.}, \textbf{P2.} and \textbf{P3.} hold provided that
    \begin{equation}\label{eq:lowerbound3}
      \dfrac{\mu}{\theta}>\dfrac{v^Tb_0}{c e_\ell^Tv}
    \end{equation}
   where $c>0$ and $v\in\mathbb{R}_{>0}^d$ verify $v^T(A+cI)\le0$.\mendth
%
%
\end{theorem}
\begin{proof}
The equivalence between the statement \eqref{st:st:3} and statement \eqref{st:st:2} follows from Lemma \ref{lem:prel} and Lemma \ref{lem:struct}. Hence, we simply have to prove the equivalence between statement \eqref{st:st:1} and statement \eqref{st:st:2}.
%

The equivalence between the Hurwitz-stability of $\sgn(S_A)$ and all the matrices in $\mathcal{Q}(S_A)$ directly follows from Lemma \ref{lem:struct}. Note that if $\mathcal{G}_{S_A}$ has a cycle, then there exists at least an unstable matrix in $\mathcal{Q}(S_A)$. Let us assume then for the rest of the proof that there is no cycle in $\mathcal{G}_{S_A}$ and let us focus now on the statement equivalent to the output controllability condition. From Lemma \ref{lem:struct}, we know that since all the matrices in $\mathcal{Q}(S_A)$ are Hurwitz stable, then the graph $\mathcal{G}_{S_A}$ is an acyclic directed graph and $S_A$ has negative diagonal entries. From Lemma \ref{lem:prel}, we know that the network is output controllable if and only if there is a path from node 1 to node $\ell$ in the graph $\mathcal{G}_A$. To algebraically formulate this, we introduce the sign matrix $S_C\in\mathbb{S}^{d\times d}$ for which the associated graph $\mathcal{G}_{S_C}:=(\mathcal{V},\mathcal{E}\cup(\ell,1))$, where $\ell\ne 1$ by assumption, consists of the original graph to which we add an edge from node $\ell$ to node $1$. Note that if $(\ell,1)\in \mathcal{E}$, then $S_A=S_C$. The output controllability condition then equivalently turns into the existence of a cycle in the graph $\mathcal{G}_{S_C}$ (recall the no cycle assumption for $\mathcal{G}_{S_A}$ as, otherwise, some matrices in $\mathcal{Q}(S_A)$ would not be Hurwitz stable). Considering again Lemma \ref{lem:prel}, we can turn the existence condition of a cycle in $\mathcal{G}_{S_C}$ into an instability condition for some of the matrices in $\mathcal{Q}(S_C)$. Since $S_C$ is a Metzler sign-matrix, then there exist some unstable matrices in $\mathcal{Q}(S_C)$ if and only if  $v^T\sgn(S_C)\nless0$ for all $v>0$. Using Farkas' lemma \cite{Boyd:04}, this is equivalent to saying that there exist $v_2,v_3\in\mathbb{R}_{\ge0}^{d}$, $||v_2+v_3||_1=1$ such that $v_2-\sgn(S_C)v_3=0$. Therefore, the existence of $v_2,v_3$ verifying the conditions above is equivalent to saying  that for all $A\in\mathcal{Q}(S_A)$, there exists a $w\ge0$, $w_1>0$, such that $w^TA+e_\ell^T=0$. Noting, finally, that $\sgn(S_C)=\sgn(\sgn(S_A)+e_\ell e_1^T)$ yields the result.
\end{proof}

\begin{remark}
  In the case $\ell=1$, the output controllability condition is trivially satisfied as the actuated species coincides with the controlled species and hence we only need to check the Hurwitz stability condition $v^T\sgn(S_A)<0$ for some $v\in\mathbb{R}_{>0}^d$.
\end{remark}

As in the nominal and robust cases, the above result can also be naturally reformulated as the  linear feasibility problem:
  \begin{equation}\label{feas:struct}
    \begin{array}{rcl}
      \textnormal{Find} && v_1\in\mathbb{R}_{>0}^d,\ v_2,v_3\in\mathbb{R}_{\ge0}^{d}\\
      \textnormal{s.t.}&&v_1^T\sgn(A)<0\\
      &&\mathds{1}_d^T(v_2+v_3)=1\\
      &&v_2-\sgn(\sgn(S_A)+e_1e_\ell^T)v_3=0
    \end{array}
  \end{equation}
  where $\mathds{1}_d$ the $d$-dimensional vector of ones. The computational complexity of this program is slightly higher (i.e. $3d$ variables, $4d$ inequality constraints and $2d+1$ equality constraints) but is still linear in $d$ and, therefore, this program will remain tractable even when $d$ is large.

\section{Example}\label{sec:ex}

We propose to illustrate the results by considering a variation of the stochastic switch \cite{Tian:06} described by the set of reactions given in Table \ref{table:reactions}, where the functions $f_1$ and $f_2$ are valid nonnegative functions (e.g. mass-action or Hill-type). 
Our goal is to control the mean population of $\X{2}$ by actuating $\X{1}$. We further assume that $\alpha_1,\alpha_2,\gamma_1,\gamma_2>0$, which implies that the state-space of the open-loop network is irreducible.

\textbf{Scenario 1.} In this scenario, we simply assume that $f_1$ and $f_2$ are bounded functions with respective upper-bounds $\beta_1>0$ and $\beta_2>0$. Then, using the results in \cite{Briat:13i}, the ergodicity of the network in Table \ref{table:reactions} can be established by checking the ergodicity of a comparison network obtained by substituting the functions $f_1$ and $f_2$ by their upper-bound. In the current case, the comparison network coincides with a unimolecular network with mass-action kinetics defined with
\begin{equation}\label{eq:ex1}
  A=\begin{bmatrix}
    -\gamma_1 & 0\\
    k_{12} & -\gamma_2
  \end{bmatrix}\ \textnormal{and}\ b_0=\begin{bmatrix}
  \alpha_1+\beta_1\\
  \alpha_2+\beta_2
  \end{bmatrix}.
\end{equation}
It is immediate to see that the characteristic matrix is Hurwitz stable and that the system is output controllable provided that $k_{12}\ne0$ since $e_2^TA^{-1}e_1=-k_{12}/(\gamma_1\gamma_2)$ (see Lemma \ref{lem:prel}). Hence, tracking is achieved provided that the lower bound condition \eqref{eq:lowerbound} in Theorem \ref{th:affine:nominal} is satisfied. Moreover, we can see that for any $\alpha_1,\alpha_2,\beta_1,\beta_2,\gamma_1,\gamma_2,k_{12}>0$, we have that $A\in\mathcal{Q}(S_A)$ and $b_0\in\mathcal{Q}(S_b)$ where
\begin{equation}
  S_A=\begin{bmatrix}
    \ominus & 0\\
    \oplus & \ominus
  \end{bmatrix}\ \textnormal{and}\ S_b=\begin{bmatrix}
    \oplus\\
    \oplus
  \end{bmatrix}.
\end{equation}
All the matrices in $\mathcal{Q}(S_A)$ are Hurwitz stable since the matrix $\sgn(S_A)$ is Hurwitz stable or, equivalently, since the graph associated with $S_A$ given by $\mathcal{G}_{S_A}=\{(1,2)\}$ is acyclic (Lemma \ref{lem:struct}). Moreover, this graph trivially contains a path from node 1 to node 2, proving then that tracking will be ensured by the AIC motif provided that the inequality \eqref{eq:lowerbound3} is satisfied (Theorem \ref{th:affine:structural}). Alternatively, we can prove this by augmenting the graph  $\mathcal{G}_{S_A}$ with the edge $\{(2,1)\}$ (see the proof of Theorem \ref{th:affine:structural}) to obtain the graph $\mathcal{G}_{S_C}$ with associated sign matrix
\begin{equation}
  S_C=\begin{bmatrix}
    \ominus & \oplus\\
    \oplus & \ominus
  \end{bmatrix},
\end{equation}
which is not sign-stable since the matrix $\sgn(S_C)$ has an eigenvalue located at 0 or, equivalently, since the graph $\mathcal{G}_{S_C}$ has a cycle (Lemma \ref{lem:struct}). To arrive to the same result, we can also check that the vectors $v_1=(2,1)$, $v_2=(0,0)$ and $v_3=(1/2,1/2)$ solve the linear program \eqref{feas:struct}.

\textbf{Scenario 2.} We slightly modify here the previous scenario by making the function $f_1$ affine in $X_2$, i.e. $f_1(X_2)=k_{21}X_2+\delta_1$, $k_{21},\delta_1\ge0$. It is immediate to see that the structural result fails as the resulting characteristic matrix
has the same sign pattern as the matrix $S_C$. Hence, the network is not structurally ergodic but it can be robustly ergodic. To illustrate this, we define the following intervals for the parameters $\gamma_1\in[1,2]$, $\gamma_2\in[3,4]$, $k_{12}\in[k_{12}^-,2]$ and $k_{21}=[0,k_{21}^+]$, where $0\le k_{12}^-\le2$ and $k_{21}^+\ge0$. Hence, we have that
\begin{equation}
  \begin{bmatrix}
-2 & 0\\
k_{12}^- & -4
  \end{bmatrix}=A^-\le A\le A^+= \begin{bmatrix}
-1 & k_{21}^+\\
2 & -3
  \end{bmatrix}.
\end{equation}
The stability of all the matrices $A$ in this interval is established through the stability of $A^+$ (Lemma \ref{lem:frob}), which holds provided that $k_{21}^+<2/3$. This can also be checked by verifying that $v^TA<0$ for $v=(5,k_{21}^++1)$ under the very same condition on $k_{21}^+$.  Regarding the output controllability, we need to consider the matrix $A^-$ (Theorem \ref{th:affine:robust}) and observe that if $k_{21}^-=0$ then output controllability does not hold as there is no path from node 1 to node 2 in the graph (Lemma \ref{lem:prel}). Alternatively, we can check that, in this case, $e_2^T(A^-)^{-1}e_1=0$ (Theorem \ref{th:affine:robust}) or that the linear program \eqref{feas:rob} is not feasible. To conclude, when $k_{21}^->0$ and  $k_{21}^+<2/3$, then the linear program \eqref{feas:rob} is feasible with the vectors $v=(5,k_{21}^++1)$ and $w=(k_{12}^-/8,1/4)$, proving then that, in this case, the AIC motif will ensure robust tracking for the controlled network provided that the condition \eqref{eq:lowerbound2} is satisfied. Finally, admissible ratios $\mu/\theta$ can be chosen such that the conditions \eqref{eq:lowerbound2} and \eqref{eq:lowerbound2b} are verified.

\begin{table}[h]
\centering
\caption{Modified stochastic switch of \cite{Tian:06}}\label{table:reactions}
  \begin{tabular}{|c|c|c|}
    \hline
     & Reaction & Propensity\\
     \hline
     \hline
     $\mathcal{R}_1$ & $\phib\rarrow{}\X{1}$ & $\alpha_1+f_1(X_2)$\\
     $\mathcal{R}_2$ & $\X{1}\rarrow{}\phib$ & $\gamma_1X_1$\\
     $\mathcal{R}_3$ & $\phib\rarrow{}\X{2}$ & $\alpha_2+f_2(X_1)+k_{12}X_1$\\
     $\mathcal{R}_4$ & $\X{2}\rarrow{}\phib$ & $\gamma_2X_2$\\
     \hline
  \end{tabular}
\end{table}

\section{Conclusion}

Two extensions of the nominal result (Theorem \ref{th:affine:nominal}) initially proposed in \cite{Briat:15e} have been obtained. Even if the interval and structural representations for the characteristic matrix of the network are not be necessarily exact, the resulting conditions can be exactly cast as scalable linear programs that remain applicable to networks involving a large number of distinct molecular species. A straightforward extension would be concerned with the use of mixed matrices (see e.g. \cite{Briat:14c}), which contain both interval and sign entries, in order to overcome the topological restrictions on the graph of the network imposed by structural analysis. In this case, scalable linear programming conditions can also be obtained using similar ideas. The example section demonstrates that the method is also applicable on networks with non mass-action kinetics. However, its extension to bimolecular networks is more difficult but might be possible by merging the results obtained in \cite{Briat:14c,Briat:15e} together. In this case, however, the problem will likely become NP-Hard (of complexity $O(2^d)$), but will remain applicable to networks of reasonable size. Finally, methods capturing more explicitly the exact parametric dependency, such as those in \cite{Briat:11h,Blanchini:15}, can also be considered but will necessarily lead to more complex computational problems. Their development is left for future research.

\bibliographystyle{ieeetran}

\end{document}